\documentclass[12pt]{amsproc}
\newtheorem{theorem}{\sc Theorem}[section]
\newtheorem{lemma}[theorem]{\sc Lemma}

\begin{document}

\title[Criterion for nilpotency]
{On nilpotency of higher commutator subgroups of a finite soluble group}

\author{Josean da Silva Alves}
\address{Department of Mathematics,  Federal University of Acre, Rio Branco, AC, 69915-900, Brazil}
\email{joseanalves@hotmail.com }

\author{Pavel Shumyatsky}
\address{Department of Mathematics, University of Brasilia, Brasilia, DF, 70900-100, Brazil}
\email{pavel@unb.br}
\thanks{This research was supported by DPI/UNB and FAPDF}

\maketitle

\begin{abstract} Let $G$ be a finite soluble group and $G^{(k)}$ the $k$th term of the derived series of $G$. We prove that $G^{(k)}$ is nilpotent if and only if $|ab|=|a||b|$ for any $\delta_k$-values $a,b\in G$ of coprime orders. In the course of the proof we establish the following result of independent interest: Let $P$ be a Sylow $p$-subgroup of $G$. Then $P\cap G^{(k)}$ is generated by $\delta_k$-values contained in $P$ (Lemma \ref{foca}). This is related to the so-called Focal Subgroup Theorem.
\end{abstract}

\section{Introduction}

The following sufficient condition for nilpotency of a finite group $G$ was discovered by B. Baumslag and J. Wiegold \cite{baubau}.
\medskip

{\it Let $G$ be a finite group in which $|ab|=|a||b|$ whenever the elements $a,b$ have coprime orders. Then $G$ is nilpotent.}
\medskip

Here the symbol $|x|$ stands for the order of an element $x$ in a group $G$. In \cite{BS} a similar sufficient condition for nilpotency of the commutator subgroup $G'$ was established. 
\medskip

{\it Let $G$ be a finite group in which $|ab|=|a||b|$ whenever the elements $a,b$ are commutators of coprime orders. Then $G'$ is nilpotent.}
\medskip

Of course, the conditions in both above results are also necessary for the nilpotency of $G$ and $G'$, respectively. More recently, in \cite{racai} the above results were extended as follows.

Given an integer $k\geq1$, the word $\gamma_{k}=\gamma_k(x_1,\dots,x_k)$ is defined inductively by the formulae
\[
\gamma_1=x_1,
\qquad \text{and} \qquad
\gamma_k=[\gamma_{k-1},x_k]=[x_1,\ldots,x_k]
\quad
\text{for $k\ge 2$.}
\]
The subgroup of a group $G$ generated by all values of the word $\gamma_k$ is denoted by $\gamma_k(G)$. Of course, this is the familiar $k$th term of the lower central series of $G$. If $k=2$ we have $\gamma_k(G)=G'$. The main result in \cite{racai} says that the $k$th term of the lower central series of a finite group $G$ is nilpotent if and only if $|ab|=|a||b|$ for any $\gamma_k$-values $a,b\in G$ of coprime orders. 

It was further conjectured in \cite{racai} that a similar criterion of nilpotency of the $k$th term of the derived series of $G$ holds. The purpose of the present article is to confirm the conjecture for soluble groups. 

Given an integer $k\geq0$, the word $\delta_{k}=\delta_k(x_1,\dots,x_{2^k})$ is defined inductively by the formulae
\[
\delta_0=x_1,
\qquad
\delta_k=[\delta_{k-1}(x_1,\ldots,x_{2^{k-1}}),\delta_{k-1}(x_{2^{k-1}+1},\ldots,x_{2^k})].
\]

The subgroup generated by the $\delta_k$-values in a group $G$ is the $k$th commutator subgroup of $G$, denoted by $G^{(k)}$.

\begin{theorem}\label{main} The $k$th term of the derived series of a finite soluble group $G$ is nilpotent if and only if $|ab|=|a||b|$ for any $\delta_k$-values $a,b\in G$ of coprime orders. 
\end{theorem}

It remains to be seen whether the assumption that $G$ is soluble is really necessary in this result. It seems plausible that all groups satisfying the conditions of Theorem \ref{main} are soluble.

On the other hand, it is clear that for many words $w$ a similar criterion of nilpotency of the verbal subgroup $w(G)$ cannot be established. For example, it was shown in \cite{kani} that for any $n\geq7$ the alternating group $A_n$ admits a commutator word $w$ all of whose nontrivial values have order 3. Thus, the verbal subgroup $w(G)$ need not be nilpotent even if all $w$-values have order dividing 3.

We end this short introduction by mentioning that there are several other recent results related to the theorem of Baumslag and Wiegold (see in particular \cite{bamo, ageno, motort, monakh, gumoreto, moretosaez}).

\section{Proof of Theorem \ref{main}}

Let $G$ be a finite soluble group. Recall that the group $G$ has a Sylow basis. This is a family of pairwise permutable Sylow $p_i$-subgroups $P_i$ of $G$, exactly one for each prime, and any two Sylow bases are conjugate. The basis normalizer (also known as the system normalizer) of such a Sylow basis in $G$ is $T=\bigcap _iN_G(P_i)$. The reader can consult \cite[pp. 235--240]{FiniteSolubleGroups} for general information on basis normalizers. In particular, we will use the facts that the basis normalizers are conjugate, nilpotent, and $G=T\gamma _{\infty}(G)$ where $\gamma _{\infty}(G)$ denotes the intersection of the lower central series of $G$. Moreover, basis normalizers are preserved by epimorphisms, that is, the image of $T$ in any quotient $G/N$ is a basis normalizer of $G/N$.

In the sequel we say that a subset $X$ of a group $G$ is commutator-closed if $[x,y]\in X$ for any $x,y\in X$. The set $X$ is symmetric if $X=X^{-1}$. Throughout, we use the fact that in any group the set of $\delta_k$-values is commutator-closed and symmetric.

Recall that the Fitting height of a finite soluble group is the minimal number $h=h(G)$ such that $G$ possesses a normal series of length $h$ all of whose factors are nilpotent. Throughout this article we write $G=\langle X\rangle$ to mean that $G$ is generated by the set $X$.
\begin{lemma}\label{x-clo} Any finite soluble group is generated by a commutator-closed set in which all elements have prime power order.
\end{lemma}
\begin{proof} Let $G$ be a finite soluble group of Fitting height $h$ and $\{P_1,P_2,\dots\}$ be a Sylow basis of $G$. Set $K_1=G$, $K_2=\gamma_\infty(G)$ and $K_{i+1}=\gamma_\infty(K_i)$ for $i=1,\dots,h$. Thus, $K_{h+1}=1$. Consider the Sylow basis $\{P_1\cap K_i,P_2\cap K_i, \dots\}$ in $K_i$ and let $T_i\leq K_i$ be the corresponding basis normalizer. So we have $K_i=T_iK_{i+1}$. The specific choice of the Sylow bases $\{P_1\cap K_i,P_2\cap K_i,\dots\}$ guarantees that $T_j$ normalizes $T_k$ whenever $j\leq k$. Note that we have the equalities $$G=T_1K_2=T_1T_2K_3=\dots=T_1T_2\cdots T_h.$$ Thus, $G$ is written as a product of $h$ nilpotent subgroups $T_i$ such that $T_j$ normalizes $T_k$ whenever $j\leq k$. For each $i=1,\dots,h$ write $X_i$ to denote the set of elements of prime power order contained in $T_i$. Let $X=\bigcup_{i=1}^hX_i$. We see that the set $X$ is commutator-closed and $G$ is generated by $X$. The proof is complete.
\end{proof}

\begin{lemma}\label{glav} Let $G=\langle X\rangle$ be a group generated by a commutator-closed set $X$. Then the commutator subgroup $G'$ is generated by commutators $[x_1,x_2]$, where $x_1,x_2\in X$.
\end{lemma}

\begin{proof} Let $Y=\{[x_1,x_2]\ \vert\ x_1,x_2\in X\}$ and $H=\langle Y\rangle$. Obviously, $H\leq G'$. We therefore only need to show that $G'\leq H$. It is sufficient to prove that $H$ is normal in $G$ as it will then be clear that $G/H$ is abelian. Choose $x\in X$ and let us show that the element $x$ normalizes the subgroup $H$. Choose a generator $y\in Y$ of $H$ and write $y^x=y[y,x]$. We see that both $y$ and $[y,x]$ belong to the set $Y$, whence $y^x\in H$.
We conclude that indeed $x$ normalizes $H$ and since $G=\langle X\rangle$, the subgroup $H$ is normal in $G$. The lemma follows.
\end{proof}

In the sequel we will use, sometimes implicitly, the following lemmas taken from \cite{afs}. By a normal subset of $G$ we mean a subset which is a union of conjugacy classes in $G$.

\begin{lemma}
\label{intersection}
Let $G$ be a finite group, and let $N$ be a normal subgroup of $G$.
If $P$ is a Sylow $p$-subgroup of $G$ and $X$ is a normal subset of $G$
consisting of $p$-elements, then $XN\cap PN=(X\cap P)N$.
In other words, if we use the bar notation in $G/N$, we have
$\overline{X}\cap \overline{P}=\overline{X\cap P}$.
\end{lemma}

\begin{lemma}
\label{from}
Let $G$ be a finite group, and let $P$ be a Sylow $p$-subgroup of $G$.
Assume that $N\le L$ are two normal subgroups of $G$, and use the
bar notation in $G/N$. Let $X$ be a normal subset of $G$ consisting of $p$-elements such that $\overline P\cap \overline L=\langle \overline P\cap \overline X \rangle$. Then $P\cap L=\langle P\cap X, P\cap N \rangle$.
\end{lemma}

Let $G$ be a finite group  and $P$  a Sylow $p$-subgroup of $G$. One immediate corollary of the Focal Subgroup Theorem \cite[Theorem 7.3.4]{go} is that $P\cap G'$ can be generated by commutators  lying in $P$. It is an open problem whether the similar fact holds for other commutator words. More precisely, if $w$ is a group word, we write $w(G)$ for the verbal subgroup of $G$ generated by $w$-values. The problem whether $P\cap w(G)$ can be generated by $w$-values lying in $P$ was addressed in \cite{afs}. It was shown that if $w$ is an outer commutator word, then $P\cap w(G)$ can be generated by powers of $w$-values. In what follows we will show that if $G$ is soluble, then  $P\cap G^{(i)}$ can be generated by $\delta_i$-values lying in $P$.

\begin{lemma}\label{foca} Let $i\geq1$ and $G$ be a finite soluble group. Let $P$ be a Sylow $p$-subgroup of $G$. Then $P\cap G^{(i)}$ can be generated by $\delta_i$-values lying in $P$.
\end{lemma}
\begin{proof} By Lemma \ref{x-clo} there is a commutator-closed set $X$ in $G$ such that every element in $X$ has prime power order and $G$ is generated by $X$. For each $i=0,1,\dots$ denote by $X^{(i)}$ the set of all $x\in X$ for which there are $x_1,\dots,x_{2^i}\in X$ such that $x=\delta_i(x_1,\dots,x_{2^i})$, and by $X_p^{(i)}$ the set of all $x\in X^{(i)}$ such that $x$ is of $p$-power order. Thus, $X_p^{(i)}$ is a set of $\delta_i$-values whose order is a $p$-power. It is clear that $X_p^{(i)}\subseteq X$. Denote by $Y_i$ the union of the conjugacy classes of elements from $X_p^{(i)}$. In other notation, $Y_i=(X_p^{(i)})^G$. We will show that $P\cap G^{(i)}$ is generated by $P\cap Y_i$.

First, we will use induction on $i$ to show that $G^{(i)}$ is generated by $X^{(i)}$. If $i=0$, this is obvious so assume that $i\geq1$ and $G^{(i-1)}$ is generated by $X^{(i-1)}$. Set $U=X^{(i-1)}$ and observe that $U$ is a commutator-closed set generating $H=G^{(i-1)}$. Lemma \ref{glav} tells us that the commutator subgroup $H'$ is generated by commutators $[u_1,u_2]$, where $u_1,u_2\in U$. It remains to note that the set of such commutators $[u_1,u_2]$ is precisely $X^{(i)}$ and $H'$ is precisely $G^{(i)}$.

Now assume that $G$ is a minimal counter-example to the statement that $P\cap G^{(i)}$ is generated by $P\cap Y_i$. If $N$ is a normal subgroup of $G$ and $\overline G=G/N$, by minimality we have $\overline P\cap\overline{ G^{(i)}}=\langle\overline P\cap\overline{ Y_i}\rangle$. Lemma \ref{from} tells us that $P\cap G^{(i)}=\langle P\cap Y_i, P\cap N\rangle$. In the case where $G$ possesses a normal $p'$-subgroup $N$ the result is now immediate since $P\cap N=1$. We therefore assume that $G$ does not possess nontrivial normal $p'$-subgroups.

Since $G$ is soluble, there exists a positive integer $j$ such that $G^{(j)}$ is nilpotent. In view of our assumptions this means that $G^{(j)}$ is a $p$-group. Recall that $G^{(j)}$ is generated by $X^{(j)}$. Since $G^{(j)}$ is a $p$-group, it follows that $X^{(j)}=X_p^{(j)}$. If $j\leq i$, then $G^{(i)}$ is a $p$-group generated by $X_p^{(i)}\subseteq Y_i$, as required.

Suppose that $i\leq j-1$. Use the equality $P\cap G^{(i)}=\langle P\cap Y_i, P\cap N\rangle$ with $N=G^{(j)}$. In this case $P\cap N$ is generated by $X_p^{(j)}\subseteq Y_i$ and again we obtain that $P\cap G^{(i)}=\langle P\cap Y_i\rangle$. The proof is complete. 
\end{proof}

Recall that a group $G$ is metanilpotent if it has a normal subgroup $N$ such that both $N$ and $G/N$ are nilpotent. The following lemma is well-known (see for example \cite[Lemma 3]{racai}). Here $F(K)$ denotes the Fitting subgroup of a group $K$ and $O_{p'}(K)$ stands for the maximal normal $p'$-subgroup of $K$.

\begin{lemma}\label{meta} Let $p$ be a prime and $G$ a metanilpotent group.  Suppose that $x$ is a $p$-element in $G$ such that $[O_{p'}(F(G)),x]=1$. Then $x\in F(G)$. 
\end{lemma}

Throughout the remaining part of the article $G$ denotes a finite soluble group for which there exists $k\geq1$ such that $|ab|=|a||b|$ for any $\delta_k$-values $a,b\in G$ of coprime orders.

\begin{lemma}\label{bbb} Let $x$ be a $\delta_k$-value in $G$ and $N$ a subgroup normalized by $x$. If $(|N|,|x|)=1$, then $[N,x]=1$.
\end{lemma}
\begin{proof} Choose $y\in N$ and observe that $[y,x,x]=[x^{-y},x]^x$ is a $\delta_k$-value. The order of the $\delta_k$-value $[y,x,x]$ is prime to that of $x$. Therefore we must have $|[y,x,x]x^{-1}|=|[y,x,x]||x|$. However $[y,x,x]x^{-1}=[x,y]x^{-1}[y,x]$. This is a conjugate of $x^{-1}$ and so $[y,x,x]=1$. Since $y$ here is an arbitrary element of $N$, we have $[N,x,x]=1$. Now, by a well-known property of coprime automorphisms (cf. \cite[Theorem 5.3.6]{go}) we deduce that  $[N,x]=1$.
\end{proof}

We are now in a position to complete the proof of Theorem \ref{main}.

\begin{proof}[Proof of Theorem \ref{main}] It is clear that if $G^{(k)}$ is nilpotent, then $|ab|=|a||b|$ for any $\delta_k$-values $a,b\in G$ of coprime orders. So we only need to prove the converse. Since the case where $k\leq2$ was considered in \cite{BS} and \cite{baubau}, we will assume that $k\geq2$. Without loss of generality we assume that $G$ is a counter-example with $|G^{(k)}|$ as small as possible. Since $|G^{(k)}|>|G^{(k+1)}|$, we conclude that $G^{(k+1)}$ is nilpotent and therefore $G^{(k)}$ is metanilpotent. Let $P$ be a Sylow $p$-subgroup of $G^{(k)}$. Lemma \ref{foca} tells us that $P$ is generated by $\delta_k$-values. Let $x$ be a $\delta_k$-value contained in $P$. It follows from Lemma \ref{bbb} that $[O_{p'}(F(G)),x]=1$. We deduce from Lemma \ref{meta} that $x\in F(G)$. Since this happens for every $\delta_k$-value $x$ contained in $P$, we conclude that $P\leq F(G)$. Thus, an arbitrary Sylow subgroup of $G^{(k)}$ is contained in $F(G)$ and of course this implies that $G^{(k)}\leq F(G)$. Thus, $G^{(k)}$ is nilpotent.
\end{proof}

\end{document}